\documentclass[a4paper,8pt]{article}

\usepackage{amsfonts}
\usepackage{amscd,color}
\usepackage{amsmath,amsfonts,amssymb,amscd}
\usepackage{indentfirst,graphicx,epsfig}
\usepackage{graphicx}
\usepackage{tikz}
\usetikzlibrary{calc, shapes.geometric, positioning}
\usetikzlibrary{calc, decorations.pathreplacing, calligraphy, shapes.geometric}
\input{epsf}
\usepackage{epstopdf}
\usepackage{caption}
\usepackage{mdframed}
\newtheorem{theorem}{Theorem}[section]
\newtheorem{definition}[theorem]{Definition}
\newtheorem{example}[theorem]{Example}

\newtheorem{claim}[theorem]{Claim}

\tikzstyle{snode}=[circle,draw=black,fill=white,thick, inner sep=0pt ,minimum size=1.2mm]
\tikzstyle{bnode}=[circle ,draw=black,fill=black,thick, inner sep=0pt ,minimum size=1.2mm]

\newenvironment {proof} {\noindent{\em Proof.}}{\hspace*{\fill}$\Box$\par\vspace{4mm}}
\newcommand{\ml}{l\kern-0.55mm\char39\kern-0.3mm}

\newenvironment{theorem-non}[1]{\trivlist \item [\hskip \labelsep {\bf #1}]\ignorespaces\it}{\endtrivlist}

\title{\textbf{The number of cut-edges  and conflict-free connection number in planar graphs}}

\author{Pham Hoang Ha\footnote{E-mail address: ha.ph@hnue.edu.vn .}\\
	Department of Mathematics\\
	School of Mathematics and Computer Science\\
	Hanoi National University of Education\\
	136 XuanThuy Str., Hanoi, Vietnam\\
	\medskip\\    
	Dang Dinh Hanh \footnote{E-mail address: hanhdd@hau.edu.vn (Corresponding author)}\\
	Department of Mathematics\\
	Hanoi Architectural University\\
	129 TranPhu Str., Hanoi, Vietnam\\
	\medskip\\
	Vu Quang Minh\footnote{E-mail address: quangminhvu.nckh@gmail.com}\\
	Department of Mathematics\\
	School of Mathematics and Computer Science\\
	Hanoi National University of Education\\
	136 XuanThuy Str., Hanoi, Vietnam}%
\date{\today}

\begin{document}
\maketitle

\begin{abstract}
A \textit{cut-edge} of a connected graph is an edge whose deletion increases the number of components. In this paper, we first state some conditions for a planar graph to have a few cut-edges. After that, we use the main results to study the conditions for a colored planar graph to have a bounded conflict-free connection number.   


\end{abstract}

\section{Introduction}
In this paper, we only consider finite simple graphs. Let $G$ be a connected graph. We denote by $V(G)$, $E(G)$, $\deg_G(v)$ the vertex set, the edge set, the degree of vertex $v$ in $G$, respectively. For each subset $X\subset V(G),$ we
use $G-X$ to denote the graph obtained from $G$ by deleting the
vertices in $X$ together with their incident edges. We define $G-uv$ to be the
graph obtained from $G$ by deleting the edge $uv\in E(G)$, and
$G+uv$ to be the graph obtained from $G$ by adding an edge $uv$
between two non-adjacent vertices $u$ and $v$ of $G$. We abbreviate the set $\lbrace 1,2,\ldots ,k\rbrace$ by $[k]$. We use \cite{West2001} for terminology and notation not defined here.

A \textit{cut-edge} of a connected graph is an edge whose deletion increases the number of components. If $G$ is a connected graph of order $n$, then it is well known that $G$ contains at most $n-1$ cut-edges. The connected graphs of order $n$ with $n-1$ cut-edges must be trees. However, if the additional constraints of a connected graph are given, then the problem of determining the maximum number of cut-edges is nontrivial. Many interesting results to determine the maximum number of cut-edges in connected graphs were obtained. The range of the number of cut-edges in a connected graph of order $n$ and size $m$ was considered by Rao \cite{Rao1968}. Furthermore, problems with additional conditions on the maximum degree and minimum degree were also considered in \cite{Rao1969}. In \cite{Achuthan2003, Suil2010}, the authors determined the maximum number of cut-edges in a connected $d$-regular graph of order $n$.

Recently, the connection concepts of connected graphs that are rainbow connection, proper connection, conflict-free connection are introduced and considered. These topics having many applications in communication networks are very interesting. There are lots of papers which have been published about them. Moreover, the problems of determining the number of cut-edges of connected graphs under various additional conditions play a vital role in computing the proper connection number in \cite{Aardt2017} and the conflict-free connection number in \cite{Chang2018, Doan2024}.

In this paper, we consider some conditions for a connected planar graph to have a bounded number of cut-edges. 

The first main purpose of this paper is to study the maximum of cut-edges with the condition on the size of a connected planar graph. In particular, we prove Theorem \ref{thm1} in Section 2.

Moreover, Chang et al. \cite{Chang2018} give some degree sum conditions for a connected graph to have a few cut-edges.
\begin{theorem}(\cite{Chang2018})
	\label{thm_Chang_1}
	Let $G$ be a connected graph of order $n\geq k^2 + 4k +4, k\geq1$. If $\delta(G)\geq\frac{n-k-1}{k+2}$, then the number of cut-edges of $G$ is no more than $k$.
\end{theorem}

\begin{theorem}(\cite{Chang2018})
	\label{thm_Chang}
	Let $G$ be a connected graph of order
	\[n\geq \max\left\lbrace k^2+5k+6,\frac{\lfloor\frac{k+2}{2}\rfloor\cdot (k+2)k+k^2-k-3}{k - 2}\right\rbrace, k\geq3.\] If
	$\sigma_2(G)\geq\frac{2n-2k-3}{k+2}$, then $G$ has at most $k$ cut-edges.
\end{theorem}

We remark that when $n\ge 5$, the conditions $n \ge k^2+4k+4$ and $\delta(G)\geq\frac{n-k-1}{k+2}$ in Theorem \ref{thm_Chang_1}
imply that
\[
\delta(G) \ge \frac{(k+2)^2-k-1}{k+2} > k+2-1 = k+1 \ge 6.
\]
This does not hold if $G$ is planar. Thus, a natural question is whether the condition $\delta(G)\geq\frac{n-k-1}{k+2}$ remains valid in a planar graph without the condition $n \ge k^2+4k+4.$ The second main purpose of this paper, we give an affirmation for this question.

\begin{theorem} \label{thm2}
	Let $n,k \in \mathbb{N}^*$ with $ 3< n\leq 4k+11$. Let $G$ be a connected planar graph of order $n$ satisfying
	\[
	\delta(G) \ge \frac{n+k-1}{k+2}.
	\]
	Then $G$ has at most $k$ cut-edges.
\end{theorem}
By repeating the similar arguments as above, we also show that when $k \ge 4$, the conditions $n \ge k^2+5k+6$ and $\sigma_2(G)\geq\frac{2n-2k-3}{k+2}$ in Theorem \ref{thm_Chang} do not happen for a planar graph. Therefore,  we state a similar condition on the degree sum of two vertices for a connected planar graph to have a bounded number of cut-edges for the next purpose of this paper.

\begin{theorem}\label{thm3}
	Let $G$ be a connected planar graph. If $\sigma_2(G)\geq 11,$ then $G$ has at most one cut-edge.
\end{theorem}

Motivated by the above results, we also introduce several sharp conditions for a connected planar graph $G$ to have a bounded conflict-free connection number in Section 4 for the final purpose of this paper.

\section{The number of cut-edges and the size of planar graphs}

Before stating the main theorem, we first introduce a class of planar graphs.
\begin{definition}
	\upshape
	Let $H$ be a connected graph with exactly $k+1$ cut-edges $\{e_i\}, i \in [k+1]$ and  $H- \{e_i| i\in [k+1]\}$ has $k+2$ components $C_j, j \in[k+2]$ such that $|V(C_i)|= 1$ for all $i\in [k+1]$ and $C_{k+2}$ is a connected planar graph with the maximum number of edges. Set $\mathcal H$ to be the set of such graphs $H.$
\end{definition}
\begin{center}
	\begin{tikzpicture}[
		vertex/.style={circle, draw, fill=black, inner sep=0pt, minimum size=4pt},
		cutedge/.style={black, thick},
		component_ellipse/.style={ellipse, draw=black, thick, minimum width=2.5cm, minimum height=1.5cm, align=center},
		labelstyle/.style={font=\footnotesize, color=black}
		]
		\def\edgelen{2cm} 
		\node[vertex, label=above:$C_{k+1}$] (bk1) at (0,0) {};
		\node[vertex, label=above:$C_k$] (bk) at (-\edgelen, 0) {};
		\node[font=\large] (dots) at (-2*\edgelen, 0) {$\dots$};
		\node[vertex, label=above:$C_2$] (b2) at (-3*\edgelen, 0) {};
		\node[vertex, label=above:$C_1$] (b1) at (-4*\edgelen, 0) {};
		\node[component_ellipse, right=\edgelen of bk1, anchor=west] (Bk2) {\large $C_{k+2}$};
		\draw[cutedge] (bk1) -- node[above, black] {$e_{k+1}$} (Bk2.west);
		\draw[cutedge] (bk) -- node[above, black] {$e_k$} (bk1);
		\draw[cutedge] (dots) -- (bk);
		\draw[cutedge] (b2) -- (dots);
		\draw[cutedge] (b1) -- node[above, black] {$e_1$} (b2);
	\end{tikzpicture}
\begin{tikzpicture}[
	thick,
	main_node/.style={
		circle, 
		draw=black, 
		minimum size=1.8cm, 
		inner sep=2pt
	},
	dot/.style={
		circle, 
		fill=black, 
		inner sep=1.5pt, 
		outer sep=0pt
	},
	line_style/.style={draw=black} 
	]
	\node[main_node] (Center) at (0,0) {$C_{k+2}$};
	\coordinate (B1_pos) at (-1.8, 3.5);
	\draw[line_style] (Center) -- node[left] {$e_1$} (B1_pos) node[dot, label=above:$C_1$] {};
	\coordinate (B2_pos) at (-2.5, 0.5);  
	\coordinate (tail_end) at (-4.2, 1.8); 
	\draw[line_style] (Center) -- node[below] {$e_2$} (B2_pos) node[dot, label=below left:$C_2$] {};
	\draw[line_style] (B2_pos) -- (tail_end) node[dot] {};
	\coordinate (split_left) at (-2.0, -2.0); 
	\coordinate (sub_branch_end1) at (-3.5, -1.8);
	\coordinate (sub_elbow) at (-2.2, -3.2);  
	\coordinate (sub_branch_end2) at (-1.0, -3.8);
	\draw[line_style] (Center) -- (split_left) node[dot] {};
	\draw[line_style] (split_left) -- (sub_branch_end1) node[dot] {};
	\draw[line_style] (split_left) -- (sub_elbow) node[dot] {} -- (sub_branch_end2) node[dot] {};
	\coordinate (split_right) at (1.6, -1.8);
	\coordinate (Bk1_pos) at (3.8, -3.2);
	\coordinate (Dk1_pos) at (0.8, -3.5);
	\draw[line_style] (Center) -- (split_right) node[dot] {};
	\draw[line_style] (split_right) -- node[above right] {$e_{k+1}$} (Bk1_pos) node[dot, label=right:$C_{k+1}$] {};
	\draw[line_style] (split_right) -- (Dk1_pos) node[dot, label=below:$C_{k-1}$] {};
\end{tikzpicture}
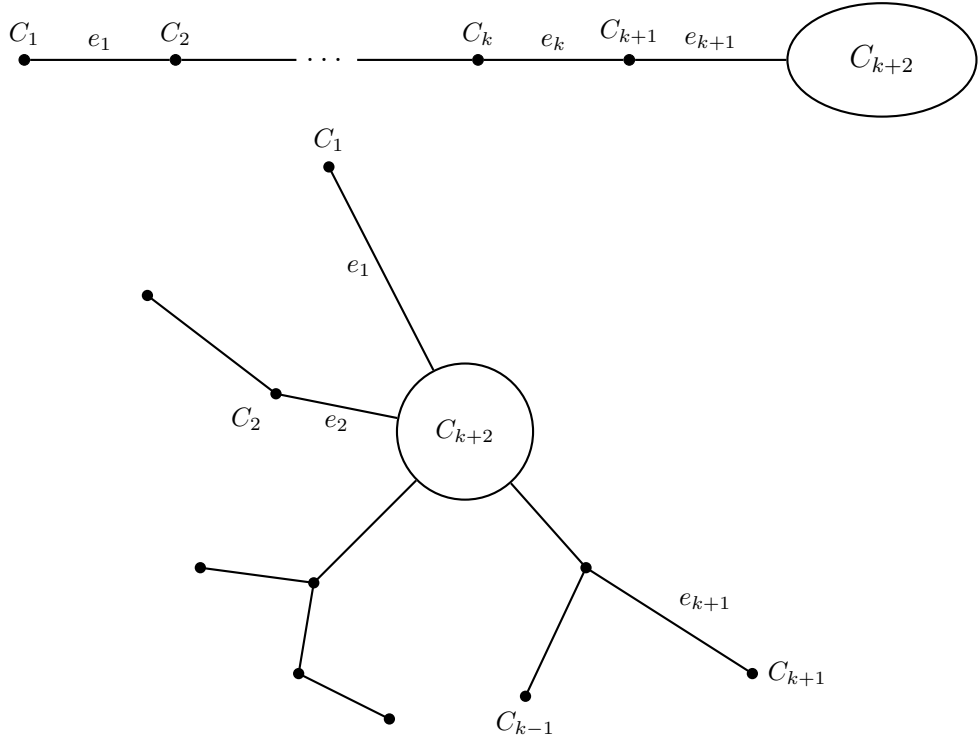
\centerline{Figure 1: Two examples of $H$ in $\mathcal{H}$}
\label{H}
\end{center}
We now have a following theorem.
\begin{theorem}\label{thm1}
	Consider two positive integers $n, k$. Let $G$ be a connected planar graph of order $n$ and $m=|E(G)|$. If $m > 3n - 2k - 9$, then $G$ has at most $k$ cut-edges unless one of the following cases holds:
	\begin{itemize}
		\item [{a)}] $G$ is isomorphic to a graph $H$ in $\mathcal H.$
		\item [{b)}] $G$ is isomorphic to a tree of order $k+2$.
		\item [{c)}] $G$ is isomorphic to a tree of order $k+3$.
	\end{itemize}
\end{theorem}

\begin{proof}
    Assume for the sake of contradiction that $G$ has strictly more than $k$ cut-edges. Let $\{e_{1}, e_{2}, \dots, e_{k+1+r}\}$ (with $r \in \mathbb{N}$) be the set of all cut-edges of $G$. Let $G \setminus \{e_{1}, \dots, e_{k+1+r}\} = \{C_{1}, C_{2}, \dots, C_{k+2+r}\}$
    be the set of connected components obtained after removing the cut-edges. Set $n_{i} = |V(C_{i})|$ for every $i \in [k+2+r]$. Here, we note that we may assert $C_{i} \not\equiv K_{2}$ for all $i \in [k+2+r]$. Indeed, otherwise, if there exists some $C_{i} = u_{i}v_{i}$, then the edge $u_{i}v_{i}$ is considered as a cut-edge of $G$. Thus, we obtain that $n_i = 1$ or $n_i \ge 3$ for every $i\in [k+2+r]$. Since $G$ is planar, each component $C_{i}$ is also planar. Therefore, if $n_i \ge 3$, we have $|E(C_{i})| \le 3n_{i} - 6$.

We consider the following three cases:

\noindent \textbf{Case 1:} $n_i \ge 3$ for all $i \in [k+2+r]$. We conclude that
\begin{align*}
|E(G)| &= \sum_{i=1}^{k+2+r}|E(C_{i})| + (k+1+r) \le \sum_{i=1}^{k+2+r}(3n_{i}-6) + k+1+r \\
&= 3\sum_{i=1}^{k+2+r}n_{i} - 6(k+2+r) + k+1+r = 3n - 5k - 11 - 5r \leq 3n - 2k - 11, 
\end{align*}
this gives a contradiction.

\noindent \textbf{Case 2:} There exist some $i, j \in [k+2+r]$ such that $n_{i} = 1$ and $n_{j} \ge 3$. Without loss of generality, we may assume that $n_1 = n_2 = \dots = n_p = 1$ and $n_j \ge 3$ for all $j \in [k+2+r] \setminus [p]$.
Hence, we have $|E(C_{i})| = 0, \forall i \in [p]$ and $|E(C_{j})| \le 3n_{j} - 6, \forall j \in [k+2+r] \setminus [p]$. Therefore, we obtain that
\begin{align*}
|E(G)| &= \sum_{i=1}^{p}|E(C_{i})| + \sum_{j=p+1}^{k+2+r}|E(C_{j})| + (k+1+r) \le 0 + \sum_{j=p+1}^{k+2+r}(3n_{j}-6) + k+1+r \\
&= 3(n-p) - 6(k+2+r-p) + k+1+r = 3n - 5k - 11 + 3p - 5r.
\end{align*}
On the other hand, by definitions we have $p \le k+1+r$ and $r \ge 0$. Then, we conclude that
\begin{align*}
|E(G)| &\le 3n - 5k - 11 + 3p - 5r \le 3n - 5k - 11 + 3(k+1+r) - 5r \\
&= 3n - 2k - 8 - 2r.
\end{align*}
If $r\geq 1,$ then $|E(G)|\leq 3n - 2k - 10.$ This contradicts the assumption of the theorem. Otherwise, thus $|E(G)|\leq 3n - 2k - 8.$ The equality holds if and only if $r=0, p=k+1$ and $|E(C_{k+2})|=3n_{k+2}-6.$ This implies that $G$ is isomorphic to a graph $H$ in $\mathcal{H}.$

\noindent \textbf{Case 3:} $n_{i} = 1$ for all $i \in [k+2+r]$. Thus, $n = k+2+r$. Hence, we have
$$|E(G)| = k+1+r = 3(k+2+r) - 2k - 5 - 2r = 3n - 2k - 5 - 2r.$$
We now consider the following three subcases:

If $r\geq 2$, then $|E(G)| =  3n - 2k - 5 - 2r\leq 3n - 2k -9.$ This is a contradiction.

If $r=1$, then $G$ is a tree with $k+2$ cut-edges. This means that $G$ is a tree of order $k+3$.

If $r=0$, then $G$ is a tree with $k+1$ cut-edges. This means that $G$ is a tree of order $k+2$.

The proof is completed.
\end{proof}

We now give an example to show that the condition of Theorem \ref{thm1} is sharp.
\begin{example}
	\upshape
	Let $T$ be a tree of order $k+4.$ Then $T$ has $k+3$ cut-edges and $|E(T)|=k+3=3|T|-2k-9.$ This shows that the result is sharp.
\end{example}

\section{The minimum degree sum and the number of cut-edges in planar graphs}
\subsection{Theorem \ref{thm2} and its sharpness}

We first prove Theorem \ref{thm2}.

\begin{proof}
	Suppose for the sake of contradiction that $G$ has at least $k+1$ cut edges. Let $C = \{e_1, e_2, \dots, e_{k+1}\}$ be $k+1$ cut-edges of $G$. Since $G$ is connected, $G \setminus C$ contains exactly $ k+2$ components. Let $G \setminus C = \{C_1, C_2, \dots, C_{k+2}\}$ and $n_i = |V(C_i)|$. Since $G$ is a planar graph, $C_i$ is planar for every $i\in [k+2].$
    \begin{claim}\label{claim}
        For two distinct indices $i,j \in [k+2],$ if there exist $u \in V(C_i), v \in V(C_j)$ such that $uv \in E(G)$, then $uv \in C$.
    \end{claim}
    \begin{proof}
        Suppose for the sake of contradiction that $uv \notin C$.
        
    \noindent\textbf{Case 1.} Suppose there exists a cut-edge $e$ in $C$ connecting $C_i$ and $C_j$. Then $G \setminus e$ is not connected. This is impossible since there exists an edge $uv$ (different from $e$) connecting $C_i$ and $C_j$.
    
    \noindent \textbf{Case 2.} Suppose there is no cut-edge connecting $C_i$ and $C_j$. Then there exists a sequence of cut-edges $e_{i_1}, e_{i_2}, \dots, e_{i_p}$ in $C$ connecting $C_i$ and $C_j$ via $C_{i_1}, C_{i_2}, \dots, C_{i_p}$. Since $C_i, C_{i_1}, C_{i_2}, \dots, C_{i_p}, C_j$ form a 2-connected subgraph in $G$, $G \setminus e_{i_1}$ remains connected (see Figure 2 for an example). This contradicts the hypothesis that $e_{i_1}$ is a cut-edge of $G$.
    \begin{center}
    \begin{tikzpicture}[
        node distance=2cm,
        bigEllipse/.style={draw, ellipse, minimum width=3.25cm, minimum height=1.5cm, align=center},
        smallCircle/.style={draw, circle, minimum size=1cm, align=center},
        vertexDot/.style={draw, circle, inner sep=1.5pt, fill=black}
    ]
        \node[bigEllipse] (Gi) at (0,0) {$C_i$};
        \node[vertexDot, label=left:$u$] (u) at ($(Gi.center) + (+0.8, 0)$) {};
        \node[bigEllipse] (Gj) at (9,0) {$C_j$};
        \node[vertexDot, label=right:$v$] (v) at ($(Gj.center) + (-0.8, 0)$) {};
        \node[smallCircle] (Gi1) at (1.5,-2.8) {$C_{i_1}$};
        \node[smallCircle] (Gi2) at (4.5,-2.8) {$C_{i_2}$};
        \node[smallCircle] (Gip) at (7.5,-2.8) {$C_{i_p}$};
        \node at (6, -2.8) {$\dots$};
        \draw (Gi) -- (Gi1) node[midway, left] {$e_{i_1}$};
        \draw (Gi1) -- (Gi2) node[midway, above] {$e_{i_2}$};
        \draw (Gip) -- (Gj) node[midway, right] {$e_{i_p}$};
        \draw (Gi2) -- (5.5, -2.8); 
        \draw (6.5, -2.8) -- (Gip); 
        \draw[dash pattern=on 25pt off 5pt] (u) -- (v) node[midway, above]{};
    \end{tikzpicture}\\
\centerline{Figure 2: A 2-connected subgraph in $G$ }
\label{P4}
    \end{center}
    \end{proof}

  By Claim \ref{claim}, we now construct a special tree $H$ with $E(H)= C$ and each component $C_i$ is considered as a vertex of $H.$ Since $H$ has at least two leaves, without loss of generality, we may assume $C_1, C_{k+2}$ are two leaves of $H.$ 
	
	If $n_1=1,$ then $\delta(G)=1.$ This is a contradiction. Hence $n_1\geq 2.$ Combining with $C_1$ is planar, we may choose a vertex $x_1\in V(C_1)$ such that $\deg_G(x_1)=\deg_{C_1}(x_1)\leq \min\{n_1-1; 5\}.$
	
	By the same argument, we also choose a vertex $x_{k+2}\in V(C_{k+2})$ such that $\deg_G(x_{k+2})=\deg_{C_{k+2}}(x_{k+2})\leq \min\{n_{k+2}-1; 5\}.$
	
	On the other hand, take each vertex $x_i \in V(C_i)$ for every $i \in \{2; ...; k+1\}.$ Thus each cut-edge contains at most two vertices in $\{x_2; x_3; ...;x_{k+1}\}$ and contains neither $x_1$ nor $x_{k+2}$. Hence, we have 
	\begin{align*}
		(k+2)\delta(G)&\leq \sum_{i=1}^{k+2}\deg_G(x_i)\leq \sum_{i=1}^{k+2}\min\{n_i-1; 5\}+2k\\
		&\leq\min\{\sum_{i=1}^{k+2}n_i-(k+2); 5(k+2)\}+2k\\
		&\leq\min\{n+ k-2; 7k+10\}.
	\end{align*}
	
	Hence
	$$ \delta(G)\leq \min\left\{\dfrac{n+k-2}{k+2}; \dfrac{7k+10}{k+2}\right\}.$$
	On the other hand, combining with $\delta(G)\leq 5$ we conclude that
	
	$$ \delta(G)\leq \dfrac{n+k-2}{k+2}.$$

This gives a contradiction.	
		
	Therefore, we complete the proof of Theorem \ref{thm2}.
\end{proof}

We now give an example to show that the result of Theorem \ref{thm2} is tight. Let $G$ be a graph consisting of two triangles connected by a path such that the order of $G$ is $n = k + 6$. Then, we obtain that
	$$ \delta(G) = 2 =\dfrac{n+k-2}{k+2}.$$
	Nevertheless, $G$ has $k+1$ cut-edges. Therefore, the result of Theorem \ref{thm2} is tight.

\subsection{Theorem \ref{thm3} and its sharpness}
We now prove Theorem \ref{thm3}

\begin{proof}
	 Suppose that the assertion is false. Then $G$ has at least two cut-edges $e_1, e_2$. Set $C_i, i \in [3]$ to be $3$ components of $G\setminus \{e_1, e_2\}.$ Without loss of generality, we may assume that $C_1$ and $C_3$ have no cut-edge connecting them (see Figure 3 for an example). Hence $xy\not\in E(G)$ for every $x\in V(C_1), y\in V(C_3).$ Since $G$ is planar, $C_1+\{e_1\}, C_3+\{e_2\}$ are planar too. Furthermore, as every planar graph has at least one vertex of degree at most 5, there exist two vertices $x\in V(C_1), y\in V(C_3)$ such that $\deg_G(x)\leq 5, \deg_G(y)\leq 5.$ Hence $\sigma_2(G)\leq \deg_G(x)+ \deg_G(y)\leq 10.$ This gives a contradiction. Therefore, the theorem holds. 
	
\end{proof}
We now introduce an example to show that the condition $\sigma_2(G)\geq 11$ is sharp. We consider three distinct planar graphs $C_i, i\in [3]$ such that $\delta(C_i)=5$ for all $i\in [3].$ Let M be the graph obtained by joining $C_1$ to $C_2$ via an edge $e_1,$ and $C_2$ to $C_3$ via an edge $e_2$  (see Figure 3). 
 It is easy to see that $M$ is planar and $\sigma_2(M) =10.$ However, $M$ has two cut-edges, which implies that the condition stated in Theorem \ref{thm3} is sharp. 

\begin{center} \label{M}
	\begin{tikzpicture}[
		component/.style={
			draw=black, 
			thick, 
			ellipse, 
			minimum width=2.5cm, 
			minimum height=1.5cm, 
			align=center
		},
		cutedge/.style={
			thick, 
		},
		vertex/.style={
			circle, 
			fill=black, 
			inner sep=0pt, 
			minimum size=4pt
		}
		]
		\node[component] (G1) at (0,0) {$C_1$};
		\node[component] (G2) [right=2cm of G1] {$C_2$};
		\node[component] (G3) [right=2cm of G2] {$C_3$};
		\node[vertex] (v1) at (G1.east) {};
		\node[vertex] (v2_left) at (G2.west) {};
		\node[vertex] (v2_right) at (G2.east) {};
		\node[vertex] (v3) at (G3.west) {};
		\draw[cutedge] (v1) -- (v2_left) node[midway, above] {$e_1$};
		\draw[cutedge] (v2_right) -- (v3) node[midway, above] {$e_2$};
	\end{tikzpicture}
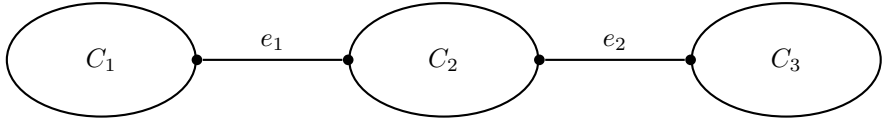
\centerline{Figure 3: $M$}
\end{center}

\section{$cfc(G)$ of planar graph}
A graph $G$ is called an edge-colored graph (or briefly a colored graph) if each of its edges is assigned to one color. The concept of a \emph{conflict-free connection} was initially presented by Czap et al. \cite{Czap}. A path within a graph $G$ is labeled \emph{conflict-free} if there exists a color appearing exactly once in its edges. A graph $G$ is called \emph{conflict-free connected} if every two vertices in $V(G)$ can be connected by a conflict-free path. The minimum number of colors required to make a connected graph $G$ conflict-free connected is referred to as the \emph{conflict-free connection number} of $G$, denoted by $cfc(G)$. Many results on the conflict-free connection number of a connected colored graph have been stated: If $G$ is a non-complete, $2$-connected graph, then the authors in \cite{Czap} showed that $cfc(G)=2$. Chang et al. \cite{Chang2021} determined the conflict-free connection number of certain tree structures. For further findings regarding this subject, readers are directed to the following references: \cite{Chang2018, Chang2019, CKP, Doan2021, Doan2024}. Furthermore, Huang et al. \cite{Huang2020} have demonstrated that determining the computational complexity of $cfc(G)$ for a given graph $G$ falls under the category of NP-hard problems.

On the other hand, there are several results for the conflict-free connection number based on various conditions of cut-edges. Let $K$ be the set of cut-edges of $G$.  Ji et al. \cite{Ji-preprint} gave an upper bound for the conflict-free connection number of a connected graph based on the number of íts cut-edges.
\begin{theorem}[Ji et al. \cite{Ji-preprint}]
	\label{upper_bound_cfc(G)}
	If $G$ is a connected, non-complete graph with the cut-edges set $K$, then $cfc(G)\leq \max\lbrace 2, \vert K\vert\rbrace$.
\end{theorem}
Moreover, let  $\lbrace K_1, K_2,\ldots, K_m\rbrace$ be components of $K$, where $m\geq 1.$ Set $h(G)=\max\{cfc(K_i)| i\in [m]\}.$ In \cite{Czap}, the authors proved that 
\begin{theorem}[Czap et al. \cite{Czap}]
	\label{cfc=cut-edges}
	If $G$ is a connected graph with cut-edges, then $h(G)\leq cfc(G)\leq h(G)+1$.
\end{theorem}

We can now combine our results in previous sections and Theorem \ref{upper_bound_cfc(G)} to present some conditions on the size and degree sum of a connected planar graph having a bound of conflict-free connection number. 
\begin{theorem}\label{thm1_1}
	Consider two positive integers $n, k$. Let $G$ be a connected planar graph of order $n$ and $m=|E(G)|$. If $m > 3n - 2k - 9$, then $cfc(G)\leq k$ unless one of the following cases holds:
	\begin{itemize}
		\item [{a)}] $G$ is isomorphic to a graph $H$ in $\mathcal H.$
		\item [{b)}] $G$ is isomorphic to a tree of order $k+2$.
		\item [{c)}] $G$ is isomorphic to a tree of order $k+3$.
		\end{itemize}
\end{theorem}
We note that if we consider $G=K_{1,k+3}$ of order $n=k+4$ and $|E(G)|=k+3=3(k+4)-2k-9,$ then $cfc(G)=k+3 > k.$ This implies that the condition of Theorem \ref{thm1_1} is tight.
\begin{theorem} \label{thm2_1}
	Let $n,k \in \mathbb{N}^*$ with $ 3< n\leq 4k+11.$ Let $G$ be a connected planar graph of order $n.$ If $
	\delta(G) \ge \frac{n+k-1}{k+2},
	$
	then $cfc(G)\leq k.$
\end{theorem}
\begin{theorem} \label{thm3_1}
	Let $G$ be a planar graph. If $\sigma_2(G)\geq 11$, then $cfc(G)=2.$ 
\end{theorem}
We end this section by giving an example to show that the result of Theorem \ref{thm3_1} is sharp. Consider $k+1$ connected planar graphs $G_1, G_2, ..., G_{k+2}$ such that $n_i = |V(G_i)|\geq 12$ and $ \delta(G_i)= 5$ for every $i\in [k+2].$ Take $x_i \in V(G_i)$ such that $\deg_{G_i}(x_i)= 5$ and $y_i\in V(G_i)\setminus\{x_i\}$ for each $i\in [k+2].$ Joining $y_{k+2}$ to $y_j$ for every $j\in [k+1]$. The resulting graph $N$ is planar and $\sigma_2(G)= 10$ (see Figure 4). Nevertheless, we obtain $cfc(G)\geq k+1$ by Theorem \ref{cfc=cut-edges}. Hence, the result of Theorem \ref{thm3_1} is sharp.
\begin{center}
	\begin{tikzpicture}[
		subgraph/.style={draw, ellipse, thick, minimum width=2.35cm, minimum height=1.65cm, align=center},
		vertex/.style={circle, fill=black, inner sep=1.5pt}
		]
		\node[subgraph] (Gk2) at (0,0) {$G_{k+2}$};
		\node[vertex, label=above:$y_{k+2}$] (yk2) at ($(Gk2.south)+(0,0.1)$) {};
		\foreach \i/\xpos in {1/-6, 2/-2, k/2, {k+1}/6} {
			\node[subgraph] (G\i) at (\xpos, -3.5) {$G_{\i}$}; 
			\node[vertex, label=below:$y_{\i}$] (y\i) at ($(G\i.north)-(0,0.1)$) {};
			\draw[thick] (yk2) -- (y\i) 
			node[pos=0.5, xshift=2.65mm, right, black] {$e_{\i}$};
		}
		\node at (0, -3.5) {$\dots$};
		\end{tikzpicture}
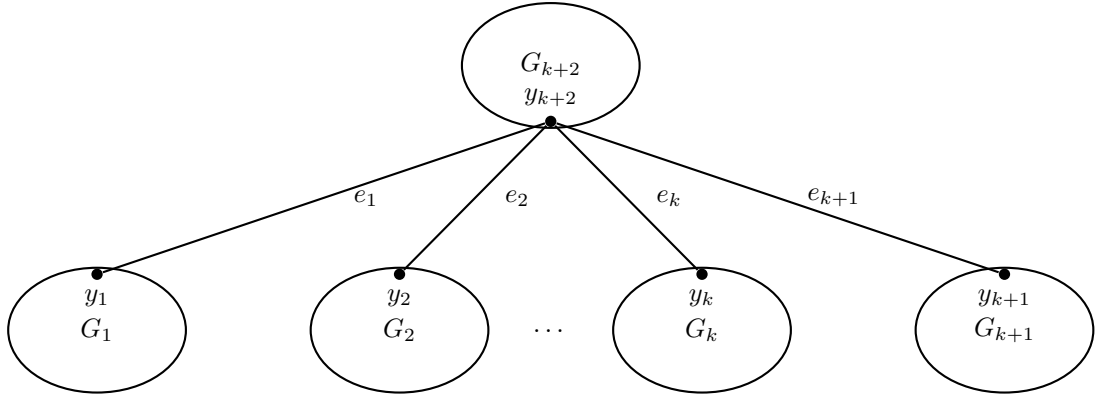
	\centerline{Figure 4: Graph $N$}
	\end{center}

\end{document}